\documentclass[12pt]{amsart}
\selectfont

\usepackage{amssymb,color, euscript, enumerate}
\usepackage{amsthm}
\usepackage{amsmath}
\usepackage{amscd}
\usepackage{txfonts}
\usepackage{comment}
\usepackage{bm}

\usepackage{amscd}

\newcommand{\Z}{\mathbb{Z}}
\newcommand{\1}{\bm{1}}

\newtheorem{thm}{Theorem}[section]
\newtheorem{lem}{Lemma}[section]
\newtheorem{prop}{Proposition}[section]
\newtheorem{cor}{Corollary}[section]
\newtheorem{rem}{Remark}[section]

\begin{document}
\title[Classification of Conformal Vectors in VOA and Automorphism]{On Classification of Conformal Vectors in Vertex Operator Algebra and the Vertex Algebra Automorphism Group}
%\title{On Uniqueness of Vertex Operator Algebra Structures and the Vertex Algebra Automorphism Group}
\author{YUTO MORIWAKI}
\date{}
\address[Y. Moriwaki]{Graduate School of Mathematical Science, The University of Tokyo, 3-8-1
Komaba, Meguro-ku, Tokyo 153-8914, Japan}
\email {moriwaki@ms.u-tokyo.ac.jp}%
\maketitle

\begin{abstract}
Herein we study conformal vectors of a $\mathbb{Z}$-graded vertex algebra of  (strong) CFT type.
We prove that the full vertex algebra automorphism group transitively acts on the set of the conformal vectors of strong CFT type
 if the vertex algebra is simple.  The statement is equivalent to the uniqueness of self-dual vertex operator algebra structures of a simple vertex algebra. %Furthermore, the set of conformal vectors of CFT type is determined.
As an application, we show that the full vertex algebra automorphism group of a simple vertex operator algebra of strong CFT type uniquely decomposes into the product of certain two subgroups and the vertex operator algebra automorphism group.
%Finally, the full vertex algebra automorphism group is determined for a vertex operator algebra with a positive-definite invariant bilinear form defined over real numbers,
%showing that the full vertex algebra automorphism group of the moonshine module $V^\natural$ over $\mathbb{R}$ is the Monster.
Furthermore, we prove that the full vertex algebra automorphism group of the moonshine module over the field of real numbers is the Monster.
\end{abstract}

\begin{comment}

We introduce two subgroups of the full vertex algebra automorphism group.

\end{comment}

% VOA vs vertex operator algebra
% vertex algebra automorphism vs full vertex algebra automorphism

\section{Introduction}

The notions of vertex algebras and their conformal vectors were originally developed by
Borcherds \cite{Bo1}.
A vertex algebra with a distinguished conformal vector is called a vertex operator algebra (VOA) \cite{FLM},
which is a mathematical axiomatization of a chiral algebra appearing in two-dimensional conformal field theory.
Different conformal vectors in a vertex algebra generally yield different VOA structures and conformal field theories.
%Thus, a fundamental issue in the theory of VOAs is the classification of conformal vectors in vertex algebras.
Thus, the classification of conformal vectors in vertex algebras is a fundamental issue in the theory of VOAs.

%In the case of the Heisenberg vertex operator algebra of rank one, which is one of the simplest non-trivial examples of a VOA,  Matsuo and Nagatomo \cite{MN} have completely determined the set of its conformal vectors.
Matsuo and Nagatomo \cite{MN} have completely determined the set of conformal vectors of the Heisenberg VOA of rank one, which is one of the simplest non-trivial examples of a VOA.
%Matsuo and Nagatomo \cite{MN} have shown that the set of its conformal vectors can be written in the form $ \{ a(1)a/2+a+\omega+Tb \}_{a,b\in M(1,0)_1}$.
In this case, the conformal vectors are parameterized by two variables. In general, vertex algebras have conformal vectors that are parameterized by infinitely many variables, as observed by Matsuo and Nagatomo \cite{MN}.
As it is difficult to express all conformal vectors explicitly,
we attempt to classify them under the action of the vertex algebra automorphism group, which suffices to classify VOA structures on a vertex algebra up to isomorphism.
%Although it is difficult to express all conformal vectors explicitly, to classify VOA structures on $V$, it is sufficient to classify conformal vectors in $V$ up to the action of $\mathrm{Aut}\,V$.
%As it is difficult to express all conformal vectors explicitly, 

To be more precise, let $V$ be a vertex algebra, and let $\mathrm{Aut}\,V$ denote the group of vertex algebra automorphisms of $V$. A conformal vector $a \in V$ is said to be  of strong CFT type if $(V,a)$ is a VOA of strong CFT type.
Here, a VOA of strong CFT type, introduced in \cite{DM}, is roughly a self-dual $\mathbb{N}$-graded VOA (see Section 2 for the precise definition).
We prove that if a vertex algebra $V$ is simple, then $\mathrm{Aut}\,V$ acts on the set of conformal vectors of strong CFT type transitively (Theorem \ref{main}). Hence, a simple vertex algebra has at most one self-dual VOA structure up to isomorphism.
Furthermore, we determine the set of conformal vectors of CFT type for a simple VOA of strong CFT type (Theorem \ref{classify}).
% We investigate the group of vertex algebra automorphisms $\mathrm{Aut}\,V$
%by considering the transitive action of  $\mathrm{Aut}\,V$ on $CV_{sCFT}$.

\begin{comment}
Note that $(V,a)$ and $(V,b)$ are isomorphic as VOAs if and only if there exists $f \in \mathrm{Aut}\,V$ such that $f(a)=b$.
To classify VOA structures on $V$,
it is sufficient to consider the action of $\mathrm{Aut}\,V$ on the set of conformal vectors in $V$.
%More specifically, set $CV_{sCFT}=\{a \in V\,|\,(V,a)\; \text{is a VOA of strong CFT type} \}$. 
More specifically, 
To be more precise, let $V$ be a vertex algebra and $a,b \in V$ be conformal vectors.
Let $\mathrm{Aut}\,V$ denote the group of vertex algebra automorphisms of $V$.
Note that $(V,a)$ and $(V,b)$ are isomorphic as VOAs if and only if there exists $f \in \mathrm{Aut}\,V$ such that $f(a)=b$.
To classify VOA structures on $V$,
it is sufficient to consider the action of $\mathrm{Aut}\,V$ on the set of conformal vectors in $V$.
%More specifically, set $CV_{sCFT}=\{a \in V\,|\,(V,a)\; \text{is a VOA of strong CFT type} \}$. 
More specifically, we show that if a vertex algebra $V$ is simple, then $\mathrm{Aut}\,V$ acts on  $CV_{sCFT}$ transitively (Theorem \ref{main}).
\end{comment}

% ???????}??????????????????}???????????????
%
As an application of Theorem \ref{main}, we investigate the structure of the full vertex algebra automorphism group.
Hereafter, $(V=\bigoplus_{n \geq 0} V_n ,\omega)$ is assumed to be a simple VOA of strong CFT type.
The VOA automorphism group $\mathrm{Aut}_\omega\,V$ is defined as the stabilizer of $\omega$ in $\mathrm{Aut}\,V$.
% VOA automorphism is studied by many researcher DG, 
% V^natural is one of the motivation 
% it is natural to ask what is the va auto of V^natural
%
VOA automorphism groups have been studied by many authors (see for example 
\cite{DG, DN}).
%the moonshine module  $(V^\natural,\omega)$ and its 
%VOA automorphism group are one of the original motivations for studying vertex algebras.
%There are many examples of a VOA whose VOA automorphism group is explicitly known \cite{DN,Sh}.
%to study the auto of .. is one of the most im problem. 
In particular, the VOA automorphism group of the moonshine module $(V^\natural,\omega)$  is the Monster, i.e., the largest sporadic finite simple group \cite{Gr,Bo1,FLM}. 
That is one of the motivations for studying vertex algebras \cite{Bo1, Bo2}.
%is an algebraic group if the vertex algebra is finitely generated \cite{DG}.
In contrast, %the full vertex algebra automorphism group $\mathrm{Aut}\,V$ is generally not a finite dimensional Lie group.
%It is difficult to treat the vertex algebra automorphism group $\mathrm{Aut}\,V$, because it is often very large. However, 
the only example of a VOA whose full vertex algebra automorphism group is explicitly known is the above-mentioned Heisenberg VOA of rank one \cite{MN}.
It is natural to ask what the full vertex algebra automorphism group of the moonshine module is.
%write what is problem, what should be solved.
In this paper, we show that the full vertex algebra automorphism group uniquely decomposes into the product of certain two subgroups and the VOA automorphism group,
and determine the full vertex algebra
automorphism groups for the real part of unitary VOAs.
%In this paper, we study the structure of the full vertex algebra automorphism group of a VOA, and determine the full vertex algebra
%automorphism groups for the real part of unitary VOAs.
In particular, the full vertex algebra automorphism group of the moonshine module $V^\natural$ over $\mathbb{R}$ is shown to be the Monster.
% for certain classes of VOAs over R.
% a family of VOAs over  whose full vertex algebra automorphism groups are determined.

%The VOA automorphism group $\mathrm{Aut}_\omega\,V$, which is defined as the stabilizer of $\omega$ in $\mathrm{Aut}\,V$, is an algebraic group if the vertex algebra is finitely generated \cite{DG}.
%In contrast, the full vertex algebra automorphism group $\mathrm{Aut}\,V$ is generally not a finite dimensional Lie group.
%It is difficult to treat the vertex algebra automorphism group $\mathrm{Aut}\,V$, because it is often very large. However, 
%The only example of a vertex algebra whose automorphism group is explicitly known is the above-mentioned Heisenberg vertex algebra of rank one \cite{MN}.
%One of the most important examples of a VOA is the moonshine module $(V^\natural,\omega)$,
%whose VOA automorphism group is the Monster group, i.e., the largest sporadic finite simple group \cite{Gr,Bo2,FLM}.
%At present, the full vertex algebra automorphism group of $V^\natural$ remains unknown.

Let us explain the basic idea. Let $pr_k$ denote the projection of $\bigoplus_{n \geq 0} V_n$ onto $V_k$.
To study the structure of the automorphism group, we introduce two subgroups of  $\mathrm{Aut}\,V$, denoted by $\mathrm{Aut}^+\,V$ and $\mathrm{Aut}^-\,V$.
%To describe our results 
%We introduce two subgroups of  $\mathrm{Aut}\,V$, denoted by $\mathrm{Aut}^+\,V$ and $\mathrm{Aut}^-\,V$.
%, which play an important role
The group $\mathrm{Aut}^- \,V$ (resp. $\mathrm{Aut}^+\,V$) consists of all vertex algebra automorphisms $f$ such that $f - \mathrm{id}$ sends $V_n$ to $\bigoplus_{k<n} V_{k}$ (resp. $\bigoplus_{k>n} V_{k}$) for all $n \in \mathbb{N}$, where $\mathrm{id}$ is the identity map.
The key observation in this paper is that if two $\Z$-gradings on a vertex algebra satisfy
certain properties, then by using the projections with respect to the $\Z$-gradings, we can construct a vertex algebra automorphism which gives an isomorphism between the $\Z$-graded vertex algebras (Lemma \ref{fundamental}).
Set  $CV^+= \{a \in \bigoplus_{n \geq 2} V_n\,|\, a \; \text{is a conformal vector and } pr_2(a)=\omega \}$. 
%Let $a \in V$ be a conformal vector.
%A conformal vector $a \in V$ gives a $\Z$-grading on $V$.
If $a \in CV^+$, then two gradings given by $a$ and $\omega$ satisfy the above properties, and a vertex algebra automorphism $\psi_{a}$ is constructed. Furthermore, we can prove that $\psi_a(\omega)=a$ and $\psi_a \in \mathrm{Aut}^+\,V$ (Theorem \ref{mor}).
In fact, there is a one-to-one correspondence between $CV^+$ and $\mathrm{Aut}^+\,V$. In particular, any conformal vector in $CV^+$ is conjugate to $\omega$ under 
$\mathrm{Aut}^+\,V$.

For a conformal vector $a \in V$, it is easy to show that $pr_1(a) \in J_1(V)=\{\;a \in V_1 \; | \; a(0)=0 \;  \}$.
%the image of $a$ under the projection $pr_1: \bigoplus_{n \geq 0} V_n \rightarrow V_1$ is contained in $J_1(V)=\{\;a \in V_1 \; | \; a(0)=0 \;  \}$.
The set $J_1(V)$ is first introduced by Dong et al. \cite{DLMM} in their study of the radical of a VOA.
Matsuo and Nagatomo have shown that if $v \in J_1(V)$, then $\exp(v(1)) \in \mathrm{Aut}\,V$ \cite{MN}. 
In this paper, we show that $\mathrm{Aut}^-\,V=\{\exp(v(1)) \,|\, v \in J_1(V)\}$ 
and $\exp(-pr_1(a)(1))a \in \bigoplus_{n \geq 2} V_n$ (Proposition \ref{Aut_m} and Lemma \ref{J1}).
Furthermore, if $a$ is of strong CFT type, then we can show that $\exp(-pr_1(a)(1))a \in CV^+$ (Lemma \ref{scft}), which proves 
that the action of $\mathrm{Aut}\,V$ on the set of conformal vectors of strong CFT type is transitive (Theorem \ref{main}). As an application, we obtain that the full vertex algebra automorphism group $\mathrm{Aut}\,V$ uniquely decomposes into the
 product of three subgroups, $\mathrm{Aut}^+\,V$,  $\mathrm{Aut}_\omega \,V$ and  $\mathrm{Aut}^-\,V$ (Theorem \ref{grp}).

%In order to understand the full vertex algebra automorphism group, we study the subgroup $\mathrm{Aut}^+\,V$.
%Let us explain the new subgroup $\mathrm{Aut}^-\,V$ in more detail. 
%Let us explain the subgroup $\mathrm{Aut}^-\,V$ first. 
%We begin by considering the set $J_1(V) = \{\;a \in V_1 \; | \; a(0)=0 \;  \}$,
%first introduced by Dong et al. \cite{DLMM} in their study of the radical of a VOA.
%It has been shown \cite{MN} that if $a \in J_1(V)$, then $\exp(a(1)) \in \mathrm{Aut}\,V$. 
%In this paper, $\mathrm{Aut}^-\,V=\{\exp(a(1))\}_{a \in J_1(V)}$ is shown  (see Proposition \ref{Aut_m}).
%Next, we consider the subgroup $\mathrm{Aut}^+\,V$. 
We assume that a VOA $V$ over $\mathbb{R}$ has a positive-definite invariant bilinear form, which is equivalent to the condition that $V \otimes_{\mathbb{R}} \mathbb{C}$ is a unitary VOA with a natural anti-involution. Then, we show that $\mathrm{Aut}^+\,V=\{1\}$ and
$\mathrm{Aut}\,V \cong \mathrm{Aut}^{-}\,V \rtimes \mathrm{Aut}_{\omega} \,V$ (Corollary \ref{uni}).
%For a conformal vector of strong CFT type $a \in V$, we construct a vertex algebra automorphism $f \in \mathrm{Aut}^-\,V$ such that $f(a)$ satisfies the properties in
%Theorem \ref{mor}, which proves 
%that the action of $\mathrm{Aut}\,V$ on the set of conformal vectors of strong CFT type is transitive (Theorem \ref{main}).
 %, where $\mathrm{Aut}^+\,V$ and $\mathrm{Aut}^-\,V$ are certain subgroups of  $\mathrm{Aut}\,V$ introduced in Section 2 
%The group $\mathrm{Aut}^{-}\,V$ is isomorphic to the certain subspace of $V_1$.
%Moreover, 
%Note that a VOA over $\mathbb{R}$ has a positive-definite invariant bilinear if and only if $V \otimes_{\mathbb{R}} \mathbb{C}$ is a unitary VOA
%, which was introduced in \cite{DL}. 
In particular, if $J_1(V)=0$, then the full vertex algebra automorphism group coincides with the VOA automorphism group.
%$\mathrm{Aut}\,V \cong \mathrm{Aut}_{\omega} \,V$.
Importantly, the space $J_1(V)$ vanishes quite often.
For example, if $V$ is a rational $C_2$-cofinite simple VOA of strong CFT type, then $J_1(V)=0$ \cite{Ma}.
%If we further assume that $V$ is rational and $C_2$-cofinite, then $\mathrm{Aut}^{-}\,V=\{1\} $, and $\mathrm{Aut}\,V \cong \mathrm{Aut}_{\omega} \,V$.
Many important VOAs (e.g., simple affine VOAs of level $k \in \mathbb{Z}_{>0}$, lattice VOAs, and the moonshine module) are unitary and satisfy $J_1(V)=0$. Hence, they provide us with many examples of VOAs over $\mathbb{R}$ whose full vertex algebra automorphism groups are equal to the VOA automorphism groups.

The remainder of this paper is organized as follows.
Section 2 reviews some definitions and introduces the notion of conformal vectors of (strong) CFT type, which are discussed throughout this paper.
Section 3 is devoted to the construction of certain automorphisms
of conformal vertex algebras that play a key role (Theorem \ref{mor}).
%Sections 4 and 5 describe the decomposition of a conformal vector with respect to the grading decomposition $V=\bigoplus_{n \in %\mathbb{N}} V_n$. 
%The homogeneous components are studied for a conformal
% vector of CFT type (Section 4) and of strong CFT type (Section 5).
Sections 4 and 5 describe the homogeneous components of a conformal vector with
respect to the grading decomposition $V =\bigoplus_{n\in \mathbb{N}} V_n$ when it is of CFT type and of strong CFT type, respectively, which proves the uniqueness of self-dual vertex operator algebra structures of a simple vertex algebra (Theorem \ref{main}).
 Section 6 proves the decomposition theorem of the automorphism group (Theorem \ref{grp}).
Finally, Section 7 is devoted to determining the full vertex algebra automorphism group of a VOA over $\mathbb{R}$ with a positive-definite invariant bilinear form, especially the moonshine VOA (Theorem \ref{pos}).

% Let \V be a $\mathbb{Z}$-graded conformal vertex algebra. Set $F^n\coloneqq \bigoplus_{k \geq n}V_k$ and let $\mathrm{pr}_n: V \rightarrow V_n$ be the natural projection. 

\section{Preliminaries and Notation}
This section provides the necessary definitions and notations for what follows.
%%%%%%%%%%%%%%%%%%%%%%%%%%%%%%%%%%%%%%%%%%%%%%%%%%%%%%%%%%%%%%%%%%%%%
\subsection{Preliminaries}\mbox{}

%All vector spaces are assumed to be over
We assume that the base field $\mathbb{K}$ is $\mathbb{R} \text{ or } \mathbb{C}$ unless otherwise stated.
For a vertex algebra $V$, we let $\bm{1}$ denote the vacuum vector of $V$ and $T$ denote the endomorphism of $V$
defined by setting $Ta=a(-2)\bm{1}$ for $a\in V$.
 A $\mathbb{Z}$-graded vertex algebra is a vertex algebra 
that is a direct sum of vector spaces $V_n$ ($n \in \Z$) such that
 $V_n(k)V_m \subset V_{n+m-k-1}$ and $\bm{1} \in V_0$.
  
%, $\mathrm{Aut}^{\geq}(\bigoplus V) \coloneqq \{\, f \in \mathrm{Aut}\, V\, |\, f(F^n) \subset F^n  \text{ for any}\, n \in \mathbb{Z} \, \}$ and
% $\mathrm{Aut}^0(\bigoplus V) \coloneqq \{\, f \in \mathrm{Aut}\, V\, |\, f(V_n) \subset V_n \text{ for any}\, n \in \mathbb{Z} \, \}$. 

A $\mathbb{Z}$-graded conformal vertex algebra $(V,\omega)$ is a $\mathbb{Z}$-graded vertex algebra $V=\bigoplus_{n \in \mathbb{Z}}V_n$ with a distinguished vector $\omega \in V_2$ satisfying the
following conditions:
the operators $L(n)=\omega(n+1)$ generate a representation of the Virasoro algebra
$$[L(n),L(m)]=(n-m)L(n+m)+(1/12)\delta_{n+m,0}(n^3-n) c,$$
where $c \in \mathbb{K}$ is a constant called the central charge;
further, $L(0)v=nv$ if $v \in V_n$; and $L(-1)=T$.
A $\mathbb{Z}$-graded conformal vertex algebra $(V,\omega)$ is said to be a {\it VOA of CFT type} if $V = \bigoplus_{n \in \mathbb{N}} V_n$, 
$\mathrm{dim}V_n$ is finite, and $V_0=\mathbb{K}\1$.
A VOA $(V,\omega)$ of CFT type is said to be of {\it strong CFT type} if $L(1)V_{1} = 0$. A simple VOA $(V,\omega)$ of CFT type is of strong CFT type precisely when $(V,\omega)$ is self-dual in the sense that the contragredient module is isomorphic to $V$ as a $V$-module.

Let $V$ be a vertex algebra.
A vector $a \in V$ is said to be a {\it conformal vector (resp. conformal vector of CFT type,
conformal vector of strong CFT type)} if 
$(V,a)$ is a $\Z$-graded conformal vertex algebra (resp. VOA of CFT type, VOA of strong CFT type).
%A vector $a \in V$ is said to be a {\it conformal vector} if $(V,a)$ is a $\mathbb{Z}$-graded conformal vertex algebra. Denote by $CV$ the set of conformal vectors. 
Let $CV$ (resp. $CV_{CFT}$, $CV_{sCFT}$) denote the set of conformal vectors (resp. the set of conformal vectors of CFT type, the set of conformal vectors of strong CFT type). 
Set $V_n^a= \{ v \in V \;|\; a(1)v=nv \;\}$ for $n \in \mathbb{K}$ and $a \in V$.
%A vector $a \in V$ is a conformal vector of CFT type if and only if it satisfies the following conditions:
A vector $a \in V$ is a conformal vector if and only if it satisfies the following conditions:

\begin{itemize}
 \item OPE
  \begin{eqnarray}
a(1)a=2a, \\
a(3)a \in \mathbb{K}\1, \\
a(n)a = 0 \text{ if } n=2 \text{ or } n \geq 4;
\end{eqnarray}
 \item derivation property
\begin{eqnarray}
a(0)= T;
\end{eqnarray}

 \item semisimplicity
  \begin{eqnarray}
V = \bigoplus_{n \in \mathbb{Z}}V_n^a.
\end{eqnarray}
\end{itemize}

A conformal vector $a$ is of CFT type if and only if it satisfies the following conditions:
\begin{eqnarray}
V = \bigoplus_{n \in \mathbb{N}}V_n^a,\\ 
\dim V_n^a < \infty \text{ for } n \in \mathbb{N}, \label{cft} \\
V_0^a = \mathbb{K}\1.
\end{eqnarray}
A vertex algebra automorphism $f$ of $V$ is a linear isomorphism of $V$ which preserves all the products:
$$f(a(n)b)=f(a)(n)f(b) \text{ for } a,b \in V \text{ and } n \in \Z$$
 with $f(\1)=\1$. Let $(V,\omega)$ be a VOA of CFT type.
 A VOA automorphism $f$ of $(V,\omega)$ is a vertex algebra automorphism of $V$ which preserves the 
conformal vector:
$$f(\omega)=\omega.$$
Let $\mathrm{Aut}\,V$ denote the group of vertex algebra automorphisms of $V$.
Then, we have the following lemma:

\begin{lem}
\label{auto_stable}
The sets $CV$, $CV_{CFT}$, and $CV_{sCFT}$ are stable under
the action of $\mathrm{Aut}\,V$.
\end{lem}

\subsection{Notation}\mbox{}

In the remainder of the paper, we assume that $V=\bigoplus_{n \in \Z} V_n$ is a $\Z$-graded vertex algebra,
and fix its grading. We will use the following symbols.

\begin{tabular}{ll}
$V=\bigoplus_{n \in \Z} V_n$ & The $\Z$-graded vertex algebra. \cr
$\mathrm{pr}_n$ &The canonical projection $\bigoplus_{i \in \Z} V_i \rightarrow V_n$. \cr
$V_{\geq n}$ &$ = \bigoplus_{i \geq n} V_i$. \cr
$\mathrm{Aut}^{+}\,V$ & $= \{\, f \in \mathrm{Aut}\, V\, |\, f(V_n) 
\subset \bigoplus_{k \geq n}V_k \text{ and}\; \mathrm{pr}_n \circ f |_{V_n}=id_{V_n} 
\text{ for all}\, n \in \mathbb{Z}  \, \}$. \cr

$\mathrm{Aut}^{-}\,V$ & $ = \{\; f \in \mathrm{Aut}\, V\; |\; f(V_n) \subset \bigoplus_{k \leq n} V_k \text{ and}\; \mathrm{pr}_n \circ f |_{V_n}=id_{V_n} \text{ for all}\; n \in \mathbb{Z}  \; \}$. \cr
$\mathrm{Aut}^{0}\,V$ & $= \{\; f \in \mathrm{Aut}\, V\; |\; f(V_n) \subset V_n\}$. \cr
%$J(V)= \{a \in V\,|\,a(0)=0 \}/\{ \mathrm{Im}\,T+\mathrm{ker}\,T\}$;
$J_1(V)$ &$ =\{a \in V_1\,|\,a(0)=0 \}$.\cr
$CV$ & The set of conformal vectors.\cr
$CV_{CFT}$ & The set of conformal vectors of CFT type. \cr
$CV_{sCFT}$ & The set of conformal vectors of strong CFT type. \cr
%$(V,\omega)$ & A conformal vertex algebra. \cr
%$CV^+$ & $=\{a \in CV\cap V_{\geq 2} \,|\, pr_2(a)=\omega \}$. \cr
\end{tabular}

\begin{comment}
$\mathrm{pr}_n: \bigoplus_{i \in \Z} V_i \rightarrow V_n \,\text{the canonical projection}$;
$V_{\geq n} = \bigoplus_{i \geq n} V_i$;
$\mathrm{Aut}^{+}\,V = \{\, f \in \mathrm{Aut}\, V\, |\, f(V_n) 
\subset \bigoplus_{k \geq n}V_k \text{ and}\; \mathrm{pr}_n \circ f |_{V_n}=id_{V_n} 
\text{ for all}\, n \in \mathbb{Z}  \, \}$;
$\mathrm{Aut}^{-}\,V = \{\; f \in \mathrm{Aut}\, V\; |\; f(V_n) \subset \bigoplus_{k \leq n} V_k \text{ and}\; \mathrm{pr}_n \circ f |_{V_n}=id_{V_n} \text{ for all}\; n \in \mathbb{Z}  \; \}$;
$\mathrm{Aut}^{0}\,V = \{\; f \in \mathrm{Aut}\, V\; |\; f(V_n) \subset V_n\}$;
%$J(V)= \{a \in V\,|\,a(0)=0 \}/\{ \mathrm{Im}\,T+\mathrm{ker}\,T\}$;
$J_1(V)=\{a \in V_1\,|\,a(0)=0 \}$.
\end{comment}

\section{Construction of Automorphisms}

In this section, we construct certain automorphisms of the $\Z$-graded vertex algebra $V$.
%These automorphisms are useful to classify conformal vectors of vertex algebras.
%Let $a \in F^n$ and $b \in F^m$ and $n, m, k$ be integers.
%Since $V_i(k)V_j \subset V_{i+j-k-1}$, the following diagram is commute.
%$$
%\begin{CD}
%\bigoplus_{i \geq n} V_i \bigoplus \bigoplus_{j \geq m} V_j
% @>\mathrm{pr}_n \bigoplus \mathrm{pr}_m >> V_n \bigoplus V_m\\
%@V\text{k-th product}VV @V\text{k-th product}VV \\
%\bigoplus_{l \geq n+m-k-1} V_l
%@>\mathrm{pr}_{n+m-k-1} >>
%V_{n+m-k-1}
%\end{CD}
%$$
%An equivalent formulation of this commutative diagram is:
The following lemma, which is clear from $V_i(k)V_j \subset V_{i+j-k-1}$, is critical to our study:
 
\begin{lem}
\label{bot}
Let $n, m, k$ be integers, $a \in V_{\geq n}$ and $b \in V_{\geq m}$. Then, $\mathrm{pr}_{n+m-k-1}(a(k)b)=\mathrm{pr}_n(a)(k)\mathrm{pr}_m(b)$.
\end{lem}

\begin{lem}
\label{subgroup}
$\mathrm{Aut}^{+}\,V$ is a subgroup of $\mathrm{Aut}\,V$.
\end{lem}

\begin{proof}
%It is clear that $\mathrm{Aut}^{0}(\bigoplus V_n)$ is a subgroup of $\mathrm{Aut}\,V$.
Let $f \in \mathrm{Aut}^{+}\,V$.
First, let us prove that $f^{-1} \in \mathrm{Aut}^{+}\,V$.
Let $v$ be a nonzero element of $V_n$.
 Set $a=f^{-1}(v)$, $a_k=\mathrm{pr}_k(a)$ for $k \in \Z$, and $k_0=\mathrm{min} \{\;k\;|\;a_k \neq 0 \;\}$.
Then, we have $v = f(f^{-1}(v))=f(\sum_{k \in \Z} a_k)=\sum_{k \geq k_0} f(a_k) \in V_n.$
Since $f \in  \mathrm{Aut}^{+}\,V$, we have $\mathrm{pr}_{k_0}(v)=\mathrm{pr}_{k_0}( \sum_{k \geq k_0} f(a_{k_0}) )
=\mathrm{pr}_{k_0}( f(a_{k_0}) )=a_{k_0} \neq 0$.
Thus, $k_0=n$, which proves $f^{-1}(v) \in V_{\geq n}$.
Furthermore, $v=\mathrm{pr}_n(f \circ f^{-1}(v))=\mathrm{pr}_n(f^{-1}(v))$.
Hence, $f^{-1} \in \mathrm{Aut}^{+}\,V$.
It is clear that $\mathrm{Aut}^{+}\,V$ is closed under products.
Hence, $\mathrm{Aut}^{+}\,V$ is a subgroup.
\end{proof}

Suppose that the $\Z$-graded vertex algebra $V=\bigoplus_{n \in \Z} V_n$ admits another $\Z$-grading,
$V = \bigoplus_{n \in \Z} V'_n$.
The following simple observation is fundamental:

\begin{lem}
\label{fundamental}
If $V'_n \subset V_{\geq n}$ for all $n \in \Z$, then there exists a vertex algebra homomorphism $\psi: V \rightarrow V$
such that $\psi(V'_n) \subset V_n$ for all $n \in \Z$.
Furthermore, if $V'_n \subset V_{\geq n}$ and $V_n \subset V'_{\geq n}$ for all $n \in \Z$, then there exists $\psi \in \mathrm{Aut}^+\,V$ such that 
$\psi(V_n) = V'_n$ for all $n \in \Z$.
\end{lem}

\begin{proof}
Let $i_n: V'_n \rightarrow V_{\geq n}$ denote the inclusion.
Define $\psi: V \rightarrow V$ by $\psi|_{V'_n}=\mathrm{pr}_n \circ i_n$.
Since $\bm{1} \in V_0 \cap V'_0$, we have $\psi(\1)=\1$.
Hence, Lemma \ref{bot} implies that $\psi$ is a vertex algebra homomorphism. 
It is clear that $\psi(V'_n) \subset V_n$.

Now, suppose that $V_n \subset V'_{\geq n}$ for all $n \in \Z$.
Then, $\psi$ is injective, since $V'_n \cap V_{\geq n+1} \subset 
V'_n \cap V'_{\geq n+1}=0$.
Let $v \in V_n$. Since $V_n \subset V'_{\geq n}$,
we can write $v = \sum_{i \geq n} v'_i$ with $v'_i \in V'_i$.
Then, $v-v'_n \in V'_{\geq n+1} \subset V_{\geq n+1}$.
Hence, $\psi(v'_n)=\mathrm{pr}_n(v'_n)=\mathrm{pr}_n(v)+\mathrm{pr}_n(v'_n-v)=v$.
Thus, $\psi$ is a vertex algebra automorphism.

We will show that $\psi \in \mathrm{Aut}^+\,V$.
Since $\mathrm{Aut}^+\,V$ is a group, it suffices to show that $\psi^{-1} \in \mathrm{Aut
}^+\,V$.
Clearly, $\psi^{-1}(V_n) \subset V_{\geq n}$.
Let $v \in V_n$, and set $\psi^{-1}(v)=w$. According to the definition of $\psi$, $\mathrm{pr}_n(w)=v$.
Hence, $\mathrm{pr}_n \circ \psi^{-1}|_{V_n}=\mathrm{id}$, which proves that $\psi \in \mathrm{Aut}^+\,V$.
\end{proof}

\begin{lem}
\label{fil_lem}
If $V'_n \subset V_{\geq n}$ and $V'_n \cap V_{\geq n+1}=0$ for all $n \in \Z$,
then $V_n \subset V'_{\geq n}$ for all $n \in \Z$.
\end{lem}

%
%?????????????????????????????????????t suffices to show ?????????????????????????????????????%Let 0 \neq v \in V????????????%Let $v$ be a nonzero element of $V_n$.
%
% set a_n=\mathrm{pr}_n(a) ??????????????????????????}?????????????????????????????????????????????????????????????%
%
%
%

\begin{proof}
Let $v$ be a nonzero element of $V_n$. Then, $v=\sum_{k \in \Z}v'_k$, $v'_k \in V'_k$.
%We define $v'_k \in V'_k$ by, and 
Set $k_0=\mathrm{min} \{\;k\;|\;v'_k \neq 0 \;\}$.
Since $V'_n \subset V_{\geq n}$, we have $\mathrm{pr}_{k_0}(v)=\mathrm{pr}_{k_0}(\sum_{k \geq k_0}v'_k)=\mathrm{pr}_{k_0}(v'_{k_0})$.
Since $V'_{k_0} \cap V_{\geq k_0+1}=0$ and $v'_{k_0} \neq 0$, we have $\mathrm{pr}_{k_0}(v'_{k_0}) \neq 0$, which implies that $k_0=n$, and $V_n \subset V'_{\geq n}$.
\end{proof}

In the remainder of the paper, we assume that  $(V,\omega)$ is a $\mathbb{Z}$-graded conformal vertex algebra, and fix its conformal vector $\omega$.
Let $\psi \in \mathrm{Aut}^{+}\,V$. Then, $\psi(\omega) \in CV \cap V_{\geq 2} $ and $\mathrm{pr}_2(\psi(\omega))=\omega$.
Set $CV^+ \coloneqq \{a \in CV \cap V_{\geq 2} \,|\,  \mathrm{pr}_2(a)=\omega  \}$.
Then, we have a map $\phi: \mathrm{Aut}^{+}\,V \rightarrow CV^+$ given by $f \mapsto f(\omega)$.
The following lemma asserts that $\phi$ is injective.
\begin{lem}
\label{auto_unique}
If $f, g\in \mathrm{Aut}^{+}\,V$ satisfy $f(\omega)=g(\omega)$, then $f=g$.
\end{lem}

\begin{proof}
Since $f^{-1}(g(\omega))=\omega$, $f^{-1}\circ g(V_n)=V_n$ for all $n \in \Z$.
According to Lemma \ref{subgroup}, $f^{-1} \circ g \in \mathrm{Aut}^{+}\,V$.
In other words, $f^{-1} \circ g|_{V_n}=id_{V_n}$ for all $n \in \Z$, which proves the lemma.
\end{proof}

The following theorem asserts that $\phi$ is surjective.

\begin{thm}
\label{mor}
Assume that $a \in V_{\geq 2}$ satisfies the following conditions:
$\mathrm{pr}_2(a)=\omega$, $a(0)=T$, $a(1)a=2a$;
$a(1)$ is semisimple on $V$; and the eigenvalues are integers. Then,
there exists a unique $\psi_a \in \mathrm{Aut}^{+}\,V$ such that $\psi_a(\omega)=a$.
In particular, $a$ is a conformal vector.
\end{thm}

\begin{cor}
\label{Aut_p}
The map $\phi: \mathrm{Aut}^{+}\,V \rightarrow CV^{+}$ , $f \mapsto f(\omega)$ is a bijection.
The inverse is $\psi$, $a \mapsto \psi_a$ , defined in Theorem \ref{mor}.
\end{cor}
%The remainder of this section is devoted to proving Theorem \ref{mor}.
Before we prove the above theorem, we need a preliminary result.
Let $a \in V_{\geq 2}$. Recall that
 $V_m^a=\{ v \in V \;|\; a(1)v=mv \;\}$ and 
$V_m^{\mathrm{pr}_2(a)}=\{ v \in V \;|\; \mathrm{pr}_2(a)(1)v=mv \;\}$ for $m \in \mathbb{K}$.
Suppose that $a(1)$ is semisimple on $V$. Since $a(1)$ maps $V_{\geq n}$ onto $V_{\geq n}$,
$V_{\geq n}$ decomposes into the eigenspaces of $a(1)$. That is, $V_{\geq n}= \bigoplus_{m \in \mathbb{K}} V_{\geq n}\cap V_m^a$.
Let $v \in V_{\geq n}$. 
According to Lemma \ref{bot}, $\mathrm{pr}_n(a(1)v)=\mathrm{pr}_2(a)(1)\mathrm{pr}_n(v)$. That is, $\mathrm{pr}_n \circ a(1)$ is equal to $\mathrm{pr}_2(a)(1)\circ \mathrm{pr}_n$ on $V_{\geq n}$.
Therefore, $\mathrm{pr}_2(a)(1)$ is semisimple on $V$, and we have a map $\mathrm{pr}_n |_{V_{\geq n} \cap V_m^a} : V_{\geq n} \cap V_m^a \rightarrow V_n \cap V_m^{\mathrm{pr}_2(a)}$. Since $\mathrm{pr}_n: V_{\geq n} \rightarrow V_n$ is surjective, $\mathrm{pr}_n |_{V_{\geq n} \cap V_m^a} : V_{\geq n} \cap V_m^a \rightarrow V_n \cap V_m^{\mathrm{pr}_2(a)} $ is also surjective.
The kernel of $\mathrm{pr}_n |_{V_{\geq n} \cap V_m^a} $ is $V_{\geq n+1}\cap V_m^a$.
Thus, we have the following:

\begin{lem}
\label{linear}
Let $a \in V_{\geq 2}$ and assume that $a(1)$ is semisimple on $V$.
Then, $\mathrm{pr}_2(a)(1)$ is semisimple.
Moreover, the projections $\mathrm{pr}_n$ induce isomorphisms $(V_{\geq n} \cap V_m^a)/(V_{\geq n+1} \cap V_m^a) \rightarrow V_n \cap V_m^{\mathrm{pr}_2(a)}$
 for all $n \in \Z$ and $m \in \mathbb{K}$. 
\end{lem}

\begin{proof}[Proof of Theorem \ref{mor}]
According to Lemma \ref{auto_unique}, $\psi_a$ is unique if it exists.
First, we show that $V_n$ and $V_n^a$ satisfy the assumptions in Lemma \ref{fundamental}.
Since $a(0)=T$, $V=\bigoplus_{n \in \Z} V_n^a$ is a $\Z$-graded vertex algebra.
By  applying Lemma \ref{linear} to $a(1)$ and using $\mathrm{pr}_2(a)=\omega$, we find that the maps $(V_{\geq n} \cap V_m^a) /(V_{\geq n+1} \cap V_m^a) \rightarrow V_n \cap V_m$ are linear isomorphisms for all $n,m \in \Z$.
If $n \neq m$, then $V_n \cap V_m=0$, which proves
 $V_n^a \subset V_{\geq n}$ and $V_n^a \cap V_{\geq n+1}=0$.
According to Lemma \ref{fil_lem} and Lemma \ref{fundamental}, there exists $\psi_a \in \mathrm{Aut}^{+}\,V$
such that $\psi_a(V_n)=V_n^a$. By $a(1)a = 2a$, we have $a \in V_2^a$.
Since $pr_2(a)=\omega$ and $a \in V_2^a$, we have $\psi_a(\omega)=a$,
which proves the theorem.
\end{proof}

\begin{lem} \label{pr2}
If $a \in CV \cap V_{\geq 2}$, then $\mathrm{pr}_2(a) \in CV \cap V_2$. Moreover, $V_n^a \cong V_n^{\mathrm{pr}_2(a)}$ as vector spaces
for all $n \in \Z$.
\end{lem}

\begin{proof}
To prove $\mathrm{pr}_2(a) \in CV$,
it suffices to show that $\mathrm{pr}_2(a)$ satisfies the conditions (1),...,(5).
According to Lemma \ref{bot}, $\mathrm{pr}_{3-n}(a(n)a)=\mathrm{pr}_2(a)(n)\mathrm{pr}_2(a)$.
This shows that the conditions (1), (2), and (3) hold. 
Using $a(0)=L(-1)$, we have the condition (4).
According to Lemma \ref{linear}, $\mathrm{pr}_2(a)(1)$ is semisimple, and the maps $(V_m^a \cap V_{\geq n})/(V_m^a \cap V_{\geq n+1}) \rightarrow V_m^{\mathrm{pr}_2(a)} \cap V_n$ are linear isomorphisms for all $n,m \in \Z$.
Thus, the condition (5) holds.
Furthermore, $V_m^{\mathrm{pr}_2(a)} \cong
\bigoplus_{n \in \Z} \,V_m^{\mathrm{pr}_2(a)}\cap V_n \cong
\bigoplus_{n \in \Z} \,(V_m^a \cap V_{\geq n})/(V_m^a \cap V_{\geq n+1}) \cong V_m^a
$ for all $m \in \Z$.
\end{proof}

\begin{cor}
\label{rem_cft}
If $a \in CV_{CFT} \cap V_{\geq 2} $, then $\mathrm{pr}_2(a) \in CV_{CFT}$.
\end{cor}

\section{Shape of Conformal Vectors of CFT type}
Hereafter, we assume that $(V,\omega)$ is a VOA of CFT type.
In this case, $V_1$ has a Lie algebra structure such that $[a, b] = a(0)b$ for $a, b \in V_1$. 
Set $\mathrm{Aut}_\omega\,V=\{f \in \mathrm{Aut}\, V \,|\, f(\omega)=\omega  \}$.
 
 \begin{lem}
\label{zero}
Let $a \in V_n$ with $n \geq 2$. If $a(0)\omega=0$, then $a \in \mathrm{Im}\,T$.
\end{lem}

\begin{proof}
By applying skew-symmetry to $a(0)\omega$, we have $
0=a(0)\omega = -Ta+TL(0)a-T^2\sum_{k \geq 2}(-1)^{k}T^{k-2}/k!L(k-1)a
\in (n-1)Ta+\mathrm{Im}\,T^2$. Since $n \neq 1$, $Ta \in \mathrm{Im}\,T^2$.
The operator $T$ is injective on $V_{\geq 1}$, since $V$ is of CFT type.
Thus, $a \in \mathrm{Im}\,T$.
\end{proof}

The Lie algebra $J_1(V) =\{a \in V_1\,|\,a(0)=0 \}$ is a subalgebra of the center of $V_1$.
The following lemma shows that
$J_1(V)$ is independent of the choice of the conformal vector of CFT type.

\begin{lem}
The map $J_1(V) \rightarrow \{a \in V\,|\,a(0)=0 \}/ \{\mathrm{ker}\,T+\mathrm{Im}\,T \}$
induced by the inclusion map $J_1(V) \rightarrow  \{a \in V\,|\,a(0)=0 \}$ is an isomorphism.
\end{lem}

\begin{proof}
Since $(V,\omega)$ is of CFT type, $\mathrm{Im}\,T \cap V_1=0$ and $\mathrm{ker}\,T=V_0$ hold, which implies that the map is injective. By combining this with Lemma \ref{zero}, it is surjective.
%we have $J_1(V) \cong \{a \in V\,|\,a(0)=0 \}/ \{\mathrm{ker}\,T+\mathrm{Im}\,T \}$.
\end{proof}

Recall that if $a \in J_1(V)$, then $a(1)$ is a locally nilpotent derivation and $\exp(a(1))$ is a vertex algebra automorphism of $V$ \cite[Lemma 1.3]{MN}.
In fact, we have the following:
\begin{lem}
\label{J1_exp}
If $a \in J_1(V)$, then $\exp(a(1)) \in \mathrm{Aut}^-\,V$.
\end{lem}

\begin{lem}
\label{J1}
Let $a \in CV$. Then $\mathrm{pr}_1(a) \in J_1(V)$ and  $\exp(-\mathrm{pr}_1(a)(1))a \in V_{\geq 2} \cap CV$.
\end{lem}

\begin{proof}
Set $a_n = \mathrm{pr}_n(a)$ for all $n \in \mathbb{N}$.
Recall that $T$ is an operator of degree one and $a_n(0)$ is of degree $n-1$. Therefore, by $T =a(0)$, we have $a_n(0)=0$ for all 
$n \neq 2$. Thus, $a_1 \in J_1(V)$. According to Lemma \ref{J1_exp}, $\exp(-a_1(1))$ is a vertex algebra automorphism.
 Hence, $\exp(-a_1(1))a \in CV$.
It remains to prove that $\exp(-a_1(1))a \in V_{\geq 2}$.
First, we prove that  $\exp(-a_1(1))(a_0+a_1+\omega)=\omega$.
Clearly, $\exp(-a_1(1))(a_0)=a_0$.
Using $a(1)a=2a$ and $a_0 \in \mathbb{K}\bm{1}$, we have
$2a_0=\mathrm{pr}_0(2a)=\mathrm{pr}_0(a(1)a)
=\mathrm{pr}_0(a_0(1)a_2+a_1(1)a_1+a_2(1)a_0+...)
=a_0(1)a_2+a_1(1)a_1+a_2(1)a_0=a_1(1)a_1 \label{J1_1}$.
This gives $\exp(-a_1(1))a_1=a_1-a_1(1)a_1=a_1-2a_0$.
By applying skew-symmetry, we have $a_1(1)\omega=L(0)a_1+TL(1)a_1+...=a_1$; thus, $\exp(-a_1(1))\omega=\omega-a_1(1)\omega+(1/2)a_1(1)^2\omega= \omega -a_1+a_0$.
Hence, $\exp(-a_1(1))(a_0+a_1+\omega)=\omega$.
According to Lemma \ref{zero} and $(a-\omega)(0)=T-T=0$, we have $a-a_0-a_1-\omega \in \mathrm{Im}\,T$.
Since $\exp(-a_1(1))$ is an automorphism,
 it preserves $\mathrm{Im}\,T$, which is a subset of $V_{\geq 2}$.
 Hence, $\exp(-\mathrm{pr}_1(a)(1))a=\exp(-\mathrm{pr}_1(a)(1))(a_0+a_1+\omega+a-a_0-a_1-\omega) \in \omega + \mathrm{Im}\,T \subset V_{\geq 2}$.
\end{proof}

%From Lemma \ref{J1}, the following lemma follows.
\begin{prop}
\label{Aut_m}
The map $J_1(V) \rightarrow \mathrm{Aut}^{-}\,V$, $a \mapsto \exp(a(1))$ is an isomorphism of groups, where  $J_1(V)$ is
considered as the additive group.
\end{prop}

\begin{proof}
Let $a, b \in J_1(V)$.
A similar computation as in the proof of Lemma \ref{J1} shows that 
$\exp(a(1))\exp(b(1))\omega=\omega+a+b+(1/2)a(1)a+a(1)b+(1/2)b(1)b=\exp((a+b)(1))\omega$.
The proof of Lemma \ref{auto_unique} also works for $\mathrm{Aut}^{-}\,V$.
Hence, $\exp(a(1))\exp(b(1))=\exp((a+b)(1))$ holds. 
%We claim that $\mathrm{Aut}^{-}\,V \rightarrow J_1(V)$, $f \mapsto \mathrm{pr}_1(f(\omega))$ is the inverse.
Since $\mathrm{pr}_1(\exp(a(1))\omega)=a$, the map is injective. Let $f \in \mathrm{Aut}^{-}\,V$. Since $\exp(-\mathrm{pr}_1(f(\omega))(1))f(\omega)=\omega$,
we have $f=\exp(\mathrm{pr}_1(f(\omega))(1))$.
%which implies $\exp(\mathrm{pr}_1(f(\omega))(1))=f$. 
Hence, the map is an isomorphism.
\end{proof}
By combining Lemma \ref{J1}, Corollary \ref{rem_cft}, and Theorem \ref{mor},
we have:
%As such, Corollary \ref{cft} follows from Lemma \ref{J1}, Corollary \ref{rem_cft}, and Theorem \ref{mor}.
\begin{cor}
\label{cft}
If $V_2 \cap CV_{CFT}=\{ \omega \}$, then all the elements in $CV_{CFT}$ are conjugate under $\mathrm{Aut}\,V$.
\end{cor}

Let $a \in V_2 \cap CV$. According to Lemma \ref{zero} and $a(0)=T=\omega(0)$,
$a-\omega \in \mathrm{Im}\,T$.
Thus, we have:

\begin{lem}
\label{V2}
Let $a \in V_2 \cap CV$. Then, there exists $h \in V_1$ such that $a=\omega+Th$.
\end{lem}

%We will not use the following result in this paper.
\begin{prop}
\label{sslie}
%Assume that the base field is $\mathbb{C}$.
If $V_1$ is $0$ or a semisimple Lie algebra, then $J_1(V)=0$ and $V_2 \cap CV_{CFT}=\{ \omega \}$.
\end{prop}

\begin{proof}
Let $a \in V_2 \cap CV_{CFT}$.
When $V_1=0$, the assertion immediately follows from Lemma \ref{V2}.
Hence, we may assume that $V_1$ is a semisimple Lie algebra.
 In this case, the center of the Lie algebra $V_1$ is $0$. 
Then, $J_1(V)=0$.
By Lemma \ref{V2}, there exists $h \in V_1$ such that $a=\omega+Th$.
Since $h(0)$ commutes with $L(0)$ and $a(1)=L(0)-h(0)$ is semisimple, the operator $h(0)$ is also semisimple.
%Since $V_1$ is a semisimple Lie algebra over $\mathbb{C}$, there exists a Cartan subalgebra $H$ of the Lie algebra $V_1$ such that
 %$h \in H$.
Since $(V, a)$ is of CFT type, the eigenvalues of $a(0)=L(0)-h(0)$ on $V_1$ must be positive integers.
If $h \neq 0$, then  $h(0)$ has both positive and negative eigenvalues, since $V_1$ is a semisimple Lie algebra.
Thus, $h=0$, and $a= \omega$.
\end{proof}
%
% ?????????????????????????????????????????????????}???
%

According to Lemma \ref{auto_stable}, $\mathrm{Aut}\,V$ acts on $CV_{CFT}$.
Furthermore, since $\mathrm{Aut}_\omega\,V$ preserves the grading, $\mathrm{Aut}_\omega\,V$ acts on $ V_2 \cap CV_{CFT}$.
Let us denote the set of orbits by $CV_{CFT}/\mathrm{Aut}\,V$ and $(V_2 \cap CV_{CFT})/ \mathrm{Aut}_\omega\,V$.
Then, we have a natural map $(V_2 \cap CV_{CFT})/ \mathrm{Aut}_\omega\,V \rightarrow CV_{CFT}/\mathrm{Aut}\,V$.
The following lemma asserts that the map is surjective:
\begin{lem}
If $a \in  V_{\geq 2} \cap CV_{CFT} $, then there exists $\psi \in \mathrm{Aut}\,V$ such that $\psi(\mathrm{pr}_2(a))=a$.
\end{lem}

\begin{proof}
Set $\omega'=\mathrm{pr}_2(a)$.
Let $b \in V_{\geq 2}$ such that $a= \omega' + Tb$ by Lemma \ref{zero}.
We denote by $b^{\omega'}_1$ the image of $b$ under the projection 
$\bigoplus_{k \geq 0} V^{\omega'}_{k} \to V^{\omega'}_1$.
According to Lemma \ref{rem_cft}, $\omega'$ is a conformal vector of CFT type.
%Thus, $Tb \in \mathrm{Im}\,T \subset V^{\omega'}_{\geq 2}$.
Thus, $a-\omega'=Tb  \in \mathrm{Im}\,T \subset V^{\omega'}_{\geq 2}$.
If $b^{\omega'}_1=0$, then $a-\omega' \in V^{\omega'}_{\geq 3}$, which implies that
 there exists $\psi \in \mathrm{Aut}\,V$ such that $\psi(\omega')=a$ by Theorem \ref{mor}.
%If $Tb \in V^{\omega'}_{\geq 3}$, then, according to ,
Thus, it suffices to show that $b^{\omega'}_1=0$.
Since $b \in V_{\geq 2}$, and $V_{\geq 2}= \bigoplus_{k \geq 0}
 V_{\geq 2} \cap V^{\omega'}_{k}$, we have $b^{\omega'}_1 \in V_{\geq 2}$.
By Lemma \ref{linear}, $b^{\omega'}_1(0)$ is semisimple on $V$.
Thus, by combining this with $b^{\omega'}_1 \in V_{\geq 2}$, we obtain $b^{\omega'}_1(0)=0$.
Let $h \in V_1$ such that $\omega'=\omega + Th$.
Since $h=\omega(1)h=(\omega'(1)-Th(1))h=\omega'(1)h$, we have $h \in V^{\omega'}_1$.
Since $b^{\omega'}_1, h \in V^{\omega'}_1$ and $b^{\omega'}_1(0)=0$, we have $0=b^{\omega'}_1(0)h=-h(0)b^{\omega'}_1$.
Hence, $b^{\omega'}_1=\omega'(1)b^{\omega'}_1=(\omega+Th)(1)b^{\omega'}_1=\omega(1)b^{\omega'}_1$.
Since $b^{\omega'}_1 \in V_{\geq 2}$, this implies that $b^{\omega'}_1=0$, as desired.
\end{proof}

\begin{cor}
\label{surj_cft}
The natural map $(V_2 \cap CV_{CFT})/ \mathrm{Aut}_\omega\,V \rightarrow CV_{CFT}/\mathrm{Aut}\,V$ is
surjective.
\end{cor}

\begin{comment}
\begin{rem}
\label{lattice}
Let $L$ be a positive-definite even lattice, $L^\vee$ the dual lattice of $L$ and $V_L$ the lattice VOA \cite{Bo1, FLM}.
Then, $J_1(V_L)=0$ follows from the definition of a lattice VOA.
Since $L\otimes_{\Z} \mathbb{C}$ is a Cartan subalgebra of the Lie algebra $(V_L)_1$, any conformal vector in $(V_L)_2 \cap CV_{CFT}$ is conjugate to 
a conformal vector of the form $\omega + T\alpha$ such that $\alpha \in L^\vee$ and  
$-(\alpha,\beta)+(1/2)(\beta,\beta) > 0$  for all $0 \neq \beta \in L $, where $( , )$ is the bilinear form on the lattice $L$.
This is equivalent to saying that the $V_L$-module $V_{L+\alpha}$ satisfies $V_{L+\alpha}=\bigoplus_{k \geq 0} (V_{L+\alpha})_{(1/2)(\alpha,\alpha)+k}$
and $\dim (V_{L+\alpha})_{(1/2)(\alpha,\alpha)}=1$, where $(V_{L+\alpha})_n=\{v \in V_{L+\alpha} \,|\, \omega(1)v=nv  \}$ . 
Hence, if $L$ is unimodular, then $(V_L)_2 \cap CV_{CFT}=\{ \omega \}$.
\end{rem}

\end{comment}

\section{Shape of Conformal Vectors of strong CFT type}
In this section, we prove our first main result (Theorem \ref{main}).

\begin{lem}
\label{pr2s}
Let $a \in CV_{sCFT} \cap V_{\geq 2}$. Then, $\mathrm{pr}_2(a) \in CV_{sCFT}$.
\end{lem}

\begin{proof}
Let $a \in CV_{sCFT} \cap V_{\geq 2}$. According to Corollary 
\ref{pr2}, $\mathrm{pr}_2(a) \in CV_{CFT}$. It remains to show that $\mathrm{pr}_2(a)(2)V_1^{\mathrm{pr}_2(a)}=0$.
According to Proposition \ref{linear}, $V_1^{\mathrm{pr}_2(a)}=\bigoplus_{k \in \mathbb{N}} \mathrm{pr}_k(V_{\geq k} \cap V_1^a)$.
Let $k \in \mathbb{N}$ and $v \in V_{\geq k} \cap V_1^a$. Since $a \in CV_{sCFT}$, $a(2)V_1^a=0$.
According to Lemma \ref{bot},
$0=\mathrm{pr}_{k-1}(a(2)v)=\mathrm{pr}_2(a)(2)\mathrm{pr}_k(v).$
Thus, $\mathrm{pr}_2(a)(2)V_1^{\mathrm{pr}_2(a)}=0$, which proves the lemma.
\end{proof}

The following lemma is a generalization of Proposition 4.8 in \cite{CKLW}.

\begin{lem}
\label{scft}
If $(V,\omega)$ is a simple VOA of strong CFT type, then
$CV_{sCFT} \cap V_2 = \{ \omega \}$.
\end{lem}

\begin{proof}
 Let $a \in CV_{sCFT} \cap V_2$.
According to Lemma \ref{V2}, we may assume that $a=\omega + Th$, where $h \in V_1$.
Suppose that $h \neq 0$.
 Since the operator $a(1)$ is semisimple, and $L(0)$ commutes with $h(0)$, $h(0)$ is also semisimple.
 Since $V$ is a simple VOA of strong CFT type, according to [Li], the Lie algebra $V_1$ has a non-degenerate invariant bilinear form $( \; , \; )$, which coincides with the first product up to scalar factor.
Let $k \in \mathbb{K}$ and $v \in V_1$ such that $h(0)v=kv$.
Then, $0=(h(0)h,v)=(h,h(0)v)=k(h,v)$.
If $k \neq 0$, then $(h,v)=0$.
Since the bilinear form is non-degenerate,
 there exists a vector $b \in V_1$ such that $h(0)b=0$ and $(h,b) \neq 0$ (i.e., $h(1)b \neq 0$). 
Since $a(1)b=(L(0)-h(0))b=b$ and $a(2)b=(L(1)-2h(1))b=-2h(1)b \neq 0$, we have
$b \in V_1^a$ and $a(2)V_1^a \neq 0$, which contradicts the assumption that $a$ is a conformal vector of strong CFT type.
 Therefore, $h=0$ and $CV_{sCFT} \cap V_2 = \{ \omega \}$.
\end{proof}

At this point, our main result follows from Lemma \ref{pr2s}, Lemma \ref{scft} , Lemma \ref{J1} and Theorem \ref{mor}.

\begin{thm}[Main Theorem]
\label{main}
If $(V, \omega)$ is a simple VOA of strong CFT type,
then all the elements in $CV_{sCFT}$ are conjugate under $\mathrm{Aut}\,V$.
\end{thm}

\section{Structure of the Vertex Algebra Automorphism group}
In this section, we prove Theorem \ref{grp} and Theorem \ref{classify}.
Recall that $\mathrm{Aut}^{0}\,V = \{f \in \mathrm{Aut}\, V \,|\, f(V_n)=V_n \text{ for all}\; n \in \mathbb{N}  \;   \}$.
Clearly,  $\mathrm{Aut}_{\omega}\,V \subset \mathrm{Aut}^{0}\,V$.

\begin{rem}
If $(V,\omega)$ is a simple VOA of strong CFT type, then $\mathrm{Aut}_{\omega}\,V = \mathrm{Aut}^{0}\,V$, for which see \cite[Corollary 4.11]{CKLW}.
\end{rem}
%we have the following:
%\begin{lem}
%\label{semi_pro}
%The group $\mathrm{Aut}^{+}\,V\rtimes \mathrm{Aut}_{\omega}\,V$ is a subgroup of $\mathrm{Aut}\,V$.
%\end{lem}

\begin{thm}
\label{grp}
For a simple VOA $(V,\omega)$ of strong CFT type, $\mathrm{Aut}\,V = \mathrm{Aut}^{-}\,V \mathrm{Aut}^{0}V \mathrm{Aut}^{+}\,V.$
More precisely, for $f \in \mathrm{Aut}\,V$, there exist unique elements $g \in \mathrm{Aut}^{-}\,V$, 
$h \in \mathrm{Aut}^{0}\,V$ and $k \in \mathrm{Aut}^{+}\,V$ such that $f=ghk$ holds. 
In particular, if $J_1(V)=0$, then $\mathrm{Aut}\,V \cong  \mathrm{Aut}^{+}\,V \rtimes \mathrm{Aut}^{0}\,V$.
\end{thm}

\begin{proof}
Let $f \in \mathrm{Aut}\,V$.
First, we show the existence of the decomposition.
Set $a=f(\omega)$ and $b=\exp(-\mathrm{pr}_1(a)(1))a$.
Since $b \in CV_{sCFT} \cap V_{\geq 2}$,  we have $b \in CV^{+}$.
Let $g=\psi_b^{-1} \circ \exp(-\mathrm{pr}_1(a)(1)) \circ f $. 
Since $\psi_b^{-1} \circ \exp(-\mathrm{pr}_1(a)(1)) \circ f(\omega) =\omega$, we have $g \in \mathrm{Aut}_{\omega}\,V$.
It is easy to check that $\psi_b \circ g = g \circ \psi_{g^{-1}(b)}$.
Hence, we obtain the decomposition.

Next, we show the uniqueness.
Let $f=\exp(h(1)) \circ g \circ \psi_a=\exp(h'(1)) \circ g' \circ \psi_{a'}$, where
$h,h' \in J_1(V)$, $g,g' \in \mathrm{Aut}_{\omega}V$ and $a,a' \in CV^{+}$.
 Computations similar to Lemma \ref{J1} show that $\mathrm{pr}_1(f(\omega))= \mathrm{pr}_1(\exp(h(1)) \circ g \circ \psi_a(\omega))=h$ holds.
Hence, $h=h'$.
 Since $\psi_a \in \mathrm{Aut}^{+}\,V$, we have $\mathrm{pr}_n \circ g \circ \psi_a|_{V_n}=g|_{V_n}$.
 This implies that $g=g'$. Finally, by $\psi_a(\omega)=a$, we have $a=a'$.
Since $\mathrm{Aut}_{\omega}\,V$ preserves the degrees, $\mathrm{Aut}\,V \cong  \mathrm{Aut}^{+}\,V \rtimes \mathrm{Aut}^{0}\,V$ 
if $J_1(V)=0$.
\end{proof}

\begin{thm}
\label{classify}
If $(V,\omega)$ is a simple VOA of strong CFT type,
then the natural map $(V_2 \cap CV_{CFT})/ \mathrm{Aut}_\omega\,V \rightarrow CV_{CFT}/\mathrm{Aut}\,V$ is a bijection.
\end{thm}

\begin{proof}
According to Corollary \ref{surj_cft}, this map is surjective.
Let $a, a' \in V_2 \cap CV_{CFT}$, satisfying $f(a)=a'$ for some $f \in \mathrm{Aut}\,V$.
By Lemma \ref{V2}, we may assume that $a=\omega + Tb$ and $a'=\omega+Tb'$, where $b, b' \in V_1$.
Let $g \in \mathrm{Aut}^{-}\,V$, 
$h \in \mathrm{Aut}^{0}\,V$, and $k \in \mathrm{Aut}^{+}\,V$, satisfying $f=ghk$.
Since $\mathrm{pr}_1(f(a))=\mathrm{pr}_1(a')=0$, we have $k=1$.
Since $a'=\mathrm{pr}_2(f(a))=\mathrm{pr}_2(gh(\omega+Tb))=\mathrm{pr}_2(g(\omega)+Tgh(b))=\omega+Th(b)$,
we have $b'=h(b)$. Thus, $h(a)=a'$, which implies that the map $(V_2 \cap CV_{CFT})/ \mathrm{Aut}_\omega\,V \rightarrow CV_{CFT}/\mathrm{Aut}\,V$ is injective.
\end{proof}

\section{Vertex Operator Algebra with Positive-Definite Invariant Form}

In this section, we examine VOAs over the real number field $\mathbb{R}$.
Let $(V, \omega)$ be a simple VOA of strong CFT type over  $\mathbb{R}$.
We assume that the invariant bilinear form $( , )$ is positive-definite.

Set $Q_n=\{v \in V_n \;|\; L(1)v=0 \}$.
Since $(V, \omega)$ is a simple VOA of strong CFT type, $V_n=Q_n \bigoplus L(-1)V_{n-1}$ holds.

The following lemma was proved in \cite[(3.16)]{Mi}.

\begin{lem}
\label{pos_l}
For all $n \in \mathbb{N}$ and $v \in V_n$,
$(\bm{1},v(2n-1)v)(-1)^n \geq 0$  and the equality holds if and only if $v=0$.
\end{lem}

For the convenience of the reader we repeat the proof from [Mi], thus making our exposition self-contained.

\begin{proof}
First, let $v \in Q_n$. Since $(v,v)=(\bm{1},v(2n-1)v)(-1)^n$ and the invariant bilinear form is positive-definite, our assertion clearly holds.
We apply induction to the degree $n$.
The assertion is clear if $n=0$ and $n=1$, since $V_1=Q_1$.
For $n \geq 2$, assume that the assertion is true for $n-1$.
Let $v \in V_n$. We may assume that $v$ takes the form $a+Tb$, where $a \in Q_n$ and $b \in V_{n-1}$. Then,
$(\bm{1},v(2n-1)v)(-1)^n =(\bm{1},a(2n-1)a+a(2n-1)Tb+Tb(2n-1)a+Tb(2n-1)Tb)(-1)^n. $
Since $a \in Q_n$, the first term on the right, viz.,  $(\bm{1},a(2n-1)a)(-1)^n$, is non-negative.
By $(\bm{1},\mathrm{Im}\,T)=0$, 
$(\bm{1},(Tb)(2n-1)a)=(\bm{1},a(2n-1)Tb-T(a(2n)Tb)+...)=(\bm{1},a(2n-1)Tb)$ holds.
Furthermore, $(\bm{1},a(2n-1)Tb)=(a,Tb)(-1)^n=0$ because $Q_n$ and $\mathrm{Im}\,T$ are orthogonal.
Finally, we have
\begin{eqnarray}
(\bm{1},Tb(2n-1)Tb)(-1)^n&=&-(2n-1)(\bm{1},[b(2n-2),T]b+T(b(2n-2)b))(-1)^n \nonumber \\
&=&(2n-1)(2n-2)(\bm{1},b(2n-3)b)(-1)^{n-1} \geq 0 \nonumber
\end{eqnarray}
 by the induction hypothesis.
Hence, $(\bm{1},v(2n-1)v)(-1)^n \geq 0$.
If the equality holds, then $(a,a)=0$ and 
$((2n-1)(2n-2)(\bm{1},b(2n-3)b)(-1)^{n-1})=0$ follow from the computations presented above. Since $n \geq 2$, again, according to the induction hypothesis, $a=0$ and $b=0$, which is the desired conclusion.
\end{proof}

%The following lemma shows that $CV$
%Recall that $CV^{+} =\{a \in CV \cap V_{\geq 2} \,|\,\mathrm{pr}_2(a)=\omega  \}$.

\begin{lem}
\label{nil}
If $CV \nsubset V_0 \oplus V_1\oplus V_2$, then there exist $n \in \mathbb{N}$ and a non-zero vector $v \in V_n$ such that
$v(2n-1)v=0$ and $n \geq 3$.
\end{lem}

\begin{proof}
Let $a \in CV \backslash  (V_0 \oplus V_1\oplus V_2)$. Set $a_k=\mathrm{pr}_k(a)$ and $n=\mathrm{max} \{\;k\;|\;a_k \neq 0 \;\}$.
Since $a \notin V_0 \oplus V_1\oplus V_2$ , $n \geq 3$.
We have $a(2n-1)a=0$, because $a$ is a conformal vector.
Hence, $0=\mathrm{pr}_0(a(2n-1)a)=a_n(2n-1)a_n$.
\end{proof}

%with Lemma \ref{pos_l} and Lemma \ref{J1}
By combining the above lemmas, we have $CV \subset V_0 \oplus V_1\oplus V_2$.
Hence, to classify conformal vectors, it suffices to consider conformal vectors in $CV \cap V_2$ by Lemma \ref{J1}. 
According to Lemma \ref{V2}, they are of the form $\omega + Th $, where $h \in V_1$
and $h(0)$ is semisimple on $V$. The following lemma shows that $h \in J_1(V)$.

%By a similar , we have:
\begin{lem}
If $h \in V_1$, $v \in V$ and $0 \neq k \in \mathbb{R}$ satisfy $h(0)v=kv$, then $v=0$. 
\end{lem}

\begin{proof}
Similarly to the proof of Lemma \ref{scft}, we have $k(v,v)=(h(0)v, v)=-(v,h(0)v)=-k(v,v)$. Hence, $v=0$.
\end{proof}

By the above argument, we have the following:
\begin{thm}
\label{pos}
Let  $(V, \omega)$ be a simple VOA of strong CFT type over  $\mathbb{R}$.
Suppose that the invariant bilinear form is positive-definite.
 Then, $CV = \{ (1/2)a(1)a+a+\omega+Tb \;|\;a, b \in J_1(V)\}$ and $\mathrm{Aut}\,V=\mathrm{Aut}^{-}\,V \rtimes \mathrm{Aut}_{\omega}\,V$. In particular, if $J_1(V)=0$, then $CV = \{ \omega \}$ and $\mathrm{Aut}\,V=\mathrm{Aut}_{\omega}\,V$.
\end{thm}

\begin{rem}
Matsuo and Nagatomo \cite{MN} have shown that the full vertex algebra automorphism group of the Heisenberg VOA of rank $1$ and level $1$ over $\mathbb{C}$ is isomorphic to  $\mathbb{C} \rtimes \Z/ 2\Z$. 
The above theorem states that the full vertex algebra automorphism group of the Heisenberg VOA
 of rank $n$ and level $1$ over $\mathbb{R}$ is $\mathbb{R}^n \rtimes O(n,\mathbb{R})$, where $O(n,\mathbb{R})$ is the orthogonal group.
%If $M^n(1,\mathbb{R})$ is , then the above theorem shows that
%$\mathrm{Aut}\,M^n(1,\mathbb{R}) \cong
\end{rem}

Let us make some remarks on the positive-definite invariant bilinear form.
A unitary VOA is a VOA over $\mathbb{C}$ with a positive-definite invariant Hermitian  form and an anti-involution $\sigma$  (see \cite{DL} for a precise definition).
If $V$ is a unitary VOA with an anti-involution $\sigma$, then $V^\sigma=\{v \in V\, |\,
\sigma(v)=v \}$ is a VOA over $\mathbb{R}$ with a positive-definite invariant bilinear form, which gives a one-to-one correspondence between unitary VOAs and VOAs over
$\mathbb{R}$ with a positive-definite invariant bilinear form (see for example \cite[Remark 5.4]{CKLW}).

 %Conversely, suppose that a VOA $V$ over $\mathbb{R}$ has a positive-definite invariant bilinear form.
%Define $\sigma :V\otimes_\mathbb{R} \mathbb{C} \rightarrow V\otimes_\mathbb{R} \mathbb{C}$ by $v\otimes(a+bi) \mapsto v\otimes(a-bi)$ for $a,b \in \mathbb{R}$ and $v \in V$.
%It is clear from the definition that $(V \otimes_\mathbb{R} \mathbb{C}, \sigma)$ is a unitary VOA.

%Combining this with Remark \ref{rational}, we have:
Hence, we have:

\begin{cor}
\label{uni}
Let $V$ be a simple unitary VOA of strong CFT type with an anti-involution $\sigma$.
Then, $\mathrm{Aut}\,(V^\sigma) \cong \mathrm{Aut}^{-}\,V^\sigma \rtimes \mathrm{Aut}_{\omega}\,V^\sigma$.
Furthermore, if $J_1(V)=0$,
then $\mathrm{Aut}\,(V^\sigma) \cong \mathrm{Aut}_{\omega}\,V^\sigma$.
\end{cor}

We remark that simple affine VOAs of level $k \in \mathbb{Z}_{>0}$, lattice VOAs and the moonshine module are unitary [DL]. They also satisfy the condition $J_1(V)=0$.

\begin{cor}
\label{mon}
The moonshine module $V^\natural$ over $\mathbb{R}$ has only one conformal vector,
and its full vertex algebra automorphism group is the Monster.
\end{cor}

\begin{rem}
\label{rational}
If $V$ is a rational $C_2$-cofinite simple VOA of strong CFT type, then $J_1(V)=0$. This follows immediately from \cite[Theorem 3]{Ma}.
Hence, if $V$ is a rational $C_2$-cofinite simple unitary VOA of strong CFT type with an anti-involution $\sigma$,
then the vertex algebra automorphism group of $V^\sigma$ coincides with the VOA automorphism group. 
\end{rem}

\section*{Acknowledgements}
The author would like to thank his research supervisor, Professor Atsushi Matsuo, for his
insightful comments and advice. The author also wishes to express his gratitude to \,\,\,\,  Hiroki Shimakura for carefully examining the manuscript and for his valuable comments. This work was supported by the Program for Leading Graduate Schools, 
MEXT, Japan.


\begin{thebibliography}{99}

  \bibitem[Bo1]{Bo1}
R.E. Borcherds, Vertex algebras, Kac-Moody algebras, and the Monster, Proc. Nat'l. Acad. Sci. USA. 83 (1986), 3068-3071.


  \bibitem[Bo2]{Bo2}
R.E. Borcherds, Monstrous moonshine and monstrous Lie superalgebras, Invent. Math. 109 (1992), 405-444.

\bibitem[CKLW]{CKLW}
S. Carpi, Y. Kawahigashi, R. Longo, and M. Weiner, From vertex operator algebras
to conformal nets and back, Mem. Am. Math. Soc. 254 (2018), no.1213, vi+85 pp. arXiv:1503.01260.

 \bibitem[DG]{DG} C. Dong and R.L. Griess Jr., Automorphism groups and derivation algebras of finitely generated vertex operator
algebras, Michigan Math. J. 50 (2002) 227-239.

\bibitem[DL]{DL} C. Dong, and X. Lin, Unitary vertex operator algebras, J. Algebra, 397 (2014), 252-277.

 \bibitem[DLMM]{DLMM}
  C. Dong, H. Li, G. Mason, and P. Montague, The radical of a vertex operator
algebra, in: Proc. of the Conference on the Monster and Lie Algebras at 
Ohio State University, May 1996, ed. by J. Ferrar and K. Harada, Walter de
Gruyter, Berlin- New York, 1998, 17-25.

\bibitem[DM]{DM}
C. Dong and G. Mason, Rational vertex operator algebras and the effective central charge, Int. Math. Res. Not. (2004), 2989-3008.


\bibitem[DN]{DN}
C. Dong and K. Nagatomo, Automorphism groups and twisted modules for lattice vertex operator algebras, Contemp. Math. 248 (1999), 117-133
 
  \bibitem[FLM]{FLM}
I. Frenkel, J. Lepowsky, and A. Meurman, Vertex Operator Algebras
and the Monster, Academic Press, Boston 1988.

  \bibitem[Gr]{Gr}
R.L. Griess Jr., The friendly giant, Invent. Math. 69 (1982), 1-102.

  \bibitem[Li]{Li}
H. Li, Symmetric invariant bilinear forms on vertex operator algebras, J. Pure Appl. Algebra 96(1994), 279-297.

\bibitem[Ma]{Ma}
G. Mason, Lattice subalgebras of strongly regular vertex operator algebras, in Conformal Field Theory, Automorphic Forms and Related Topics, Contrib. Math. Comput. Sci. 8, Springer, Heidelberg, (2014). 31-53.

 \bibitem[Mi]{Mi}
M. Miyamoto, A new construction of the moonshine vertex operator algebra over the real number field,
Ann. of Math, 159 (2004), 535-596.

\bibitem[MN]{MN}
A. Matsuo and K. Nagatomo, A note on free bosonic vertex algebra and its conformal vectors, J. Algebra 212 (1999), 395-418.

%\bibitem[Sh]{Sh}
%H. Shimakura, The automorphism group of the vertex operator algebra $V_L^+$ for an even lattice $L$ without roots. J. Algebra 280, 29–57 (2004).


\end{thebibliography}
\end{document}